\documentclass{amsart}
\usepackage{amssymb}
\usepackage{amsthm}
\usepackage[dvips]{graphicx}
\newcommand{\pn}{\par\noindent}
\newcommand{\pmn}{\par\medskip\noindent}
\newcommand{\pbn}{\par\bigskip\noindent}
\newtheorem{theor}{Theorem}[section]
\newtheorem{prop}{Proposition}[section]
\theoremstyle{definition} 
\newtheorem{defin}{Definition}[section]
\newtheorem{ex}{Example}[section] \theoremstyle{remark}
\newtheorem{rem}{Remark}[section]
\begin{document}
\title{On arithmetic of one class of plane maps}
\author{Yury Kochetkov}
\begin{abstract} We study bipartite maps on the plane with one
infinite face and one face of perimeter 2. At first we consider
the problem of their enumeration and then the connection between
the combinatorial structure of a map and the degree of its
definition field. The second problem is considered when the number
of edges is $p+1$, where $p$ is prime.
\end{abstract} \email{yukochetkov@hse.ru, yuyukochetkov@gmail.com}
\maketitle

\section{Introduction}
\pn We will study connected maps in the plane with two faces
--- one infinite, another --- of perimeter 2. Such maps can be
obtained from a plane trees by doubling one edge. For example,
\[\begin{picture}(250,40) \multiput(0,15)(20,0){6}{\circle*{3}}
\put(0,15) {\line(1,0){100}} \put(40,35){\circle*{3}}
\put(40,15){\line(0,1){20}} \put(120,12){$\Rightarrow$}

\multiput(150,15)(20,0){6}{\circle*{3}}
\put(150,15){\line(1,0){60}} \put(230,15){\line(1,0){20}}
\put(190,35){\circle*{3}} \put(190,15){\line(0,1){20}}
\qbezier(210,15)(220,30)(230,15) \qbezier(210,15)(220,0)(230,15)
\end{picture}\] All our maps will be bipartite: vertices are white
and black and adjacent vertices have different colors. The map
above generates two bipartite maps:
\[\begin{picture}(300,40) \put(0,15){\circle{4}}
\put(20,15){\circle*{3}} \put(40,15){\circle{4}}
\put(40,35){\circle*{3}} \put(60,15){\circle*{3}}
\put(80,15){\circle{4}} \put(100,15){\circle*{3}}
\put(2,15){\line(1,0){36}} \put(40,17){\line(0,1){18}}
\put(42,15){\line(1,0){18}} \put(82,15){\line(1,0){18}}
\qbezier(60,15)(70,30)(79,16) \qbezier(60,15)(70,0)(79,14)

\put(127,12){and}

\put(170,15){\circle*{3}} \put(190,15){\circle{4}}
\put(210,15){\circle*{3}} \put(210,35){\circle{4}}
\put(230,15){\circle{4}} \put(250,15){\circle*{3}}
\put(270,15){\circle{4}} \put(170,15){\line(1,0){18}}
\put(192,15){\line(1,0){36}} \put(210,15){\line(0,1){19}}
\put(250,15){\line(1,0){18}} \qbezier(231,16)(240,30)(250,15)
\qbezier(231,14)(240,0)(250,15)
\end{picture}\] In what follows such maps will be called
$2^1$-maps, i.e. maps with one face of perimeter 2 (and one
infinite face). \pmn Let $D$ be a $2^1$-map and $k_1,k_2,\ldots$
($l_1,l_2,\ldots$) be numbers of its white (black) vertices of
degrees 1, 2, and so on. The \emph{passport} of the map $D$ is an
expression
$$a_1^{k_1}a_2^{k_2}\cdot\ldots\cdot b_1^{l_1}b_2^{l_2}\cdot
\ldots,$$  where $a_1,a_2,\ldots$ and $b_1,b_2,\ldots$ are formal
variables. Thus, maps above have passports
$$a_1a_3^2b_1^2b_2b_3\text{ and } a_1^2a_2a_3b_1b_3^2.$$ In
Section 2 we will construct a generating function for $2^1$-maps
in terms of passports. \pmn Beginning from Section 3 we will
assume that $2^1$-map $D$ has $p+1$ edges, where $p$ is prime.
\begin{rem} We follow the work \cite{Zapponi}, where arithmetic of
plane bipartite trees with prime number of edges is studied.
\end{rem} \pn
Let $\beta$ be a Belyi function for $D$, $K$
--- its big field of definition ($K$ contains coordinates of all
vertices),  $\rho$ --- a prime divisor in $K$, that divides $p$,
$v_\rho$ --- the corresponding valuation and $I\subset K$ --- the
ideal of elements with positive $v_\rho$ valuation. Function
$\beta$ is called a {\it normalized model} for $2^1$-map $D$, if
\begin{itemize}
    \item some white vertex of degree $>1$ is at $0$ and some
    black vertex of degree $>1$ is at $1$;
    \item coordinates of all white vertices, except, maybe, one of
    degree 1, belong to $I$;
    \item coordinates of all black vertices, except, maybe, one of
    degree 1, belong to the set $1+I$.
\end{itemize} If besides the value of $\beta$ in white vertices is $0$
and in black --- $1$, then $\beta$ will be called a \emph{totally
normalized model}.
\begin{rem} In work \cite{Zapponi} a Shabat $p$ polynomial for a tree
$T$ is called normalized, if
\begin{itemize}
    \item some white vertex is in zero and some black --- in $1$;
    \item coordinates of all white vertices belong to $I$;
    \item coordinates of all black vertices belong to the set
$1+I$.
\end{itemize} \end{rem}
\pn In Section 3 the existence theorem (Theorem 3.1) is proved:
each $2^1$-map with $p+1$ edges has a normalized model. \pmn
Arithmetic of normalized model is investigated in Section 5. Let a
$2^1$-map $D$ has $n+1$ white vertices $v_0,\ldots,v_n$ with
coordinates $x_0,\ldots,x_n$ and $m+1$ black vertices
$u_0,\ldots,u_m$ with coordinates $y_0,\ldots,y_m$. Let
$e=v_\rho(p)$ be the ramification index. The following statement
holds (Theorem 5.1): let $\beta$ be a normalized model, then there
are three cases.
\begin{itemize}
    \item Coordinate $x_s$ of one white vertex $v_s$ of
    degree $1$ does not belong to $I$. In this case
    $v_\rho(x_i)=e/(n-1)$, $i\neq s$, and
    $v_\rho(x_i-x_j)=e/(n-1)$, $i,j\neq s$.
    \item $v_\rho(x_i)=e/(n+1)$ for all $i$ and
    $v_\rho(x_i-x_j)=e/(n+1)$ for all $i\neq j$.
    \item $v_\rho(x_i)=l>0$ for all white vertices, except one
    $v_s$ of degree $1$. Here $v_\rho(x_s)=k>0$, $k<l$,
    $l=(e-2k)/(n-1)$, and $v_\rho(x_i-x_j)=l$ for $i,j\neq s$.
\end{itemize}
An analogous result is valid for black vertices.
\begin{rem} In work \cite{Zapponi} the result is as follows: if
a tree $T$ with prime number of edges has $n+1$ white vertices and
$p$ is its normalized polynomial, then
\begin{itemize}
    \item $v_\rho(x_i)=e/n$ for all $i$;
    \item $v_\rho(x_i-x_j)=e/n$ for all $i,j$, $i\neq j$.
\end{itemize} \end{rem} \pn
\textbf{Corollary.} If vertices $v_i$ and $v_j$ both have degree
$>1$ then
$$v_\rho(x_i-x_j)=e/(n-1)\text{ or } e/(n+1) \text{ or }
(e-2k)/(n-1)>k.$$ In Section 6 (Theorem 6.1) we prove the
existence and the uniqueness of the \emph{canonical model}, i.e.
Belyi function $\beta$ with the following properties:
\begin{itemize}
    \item the value of $\beta$ in white vertices is $0$
     and in black --- $1$;
    \item the sum of coordinates of white vertices of degree $>1$
    is $0$;
    \item the sum of coordinates of black vertices of degree $>1$
    is $1$.
\end{itemize}
\begin{rem} In \cite{Zapponi} the canonical model is such Shabat
polynomial that: the sum of coordinates of \emph{all} white
vertices is $0$ and the sum of coordinates of \emph{all} black
vertices is $1$. \end{rem} We obtain the canonical model from a
normalized model with the use of some coordinate change that
preserves the big definition field and the following property: if
vertices $v_i$ and $v_j$ both have degree $>1$ then
$$v_\rho(x_i-x_j)=e/(n-1)\text{ or } e/(n+1) \text{ or }
(e-2k)/(n-1)>k.$$ The field of definition $L\subset K$ of a map
$D$ coincides with the field of definition of its canonical model.
Let $\tau$ be a prime divisor in $L$ that divide $p$ and is
divided by $\rho$, $v_\tau$ be the corresponding valuation in $L$
and $e_\tau=v_\tau(p)$ --- ramification index. \pmn Let $k$ be the
maximal degree of white vertices of map $D$ and $s_i$,
$i=1,\ldots,k$, --- the number of white vertices of degree $i$.
Let
$$s=\text{gcd}(\,s_is_j,\, 2\leqslant i<j\leqslant k,\,
s_i(s_i-1),\, 2\leqslant i\leqslant k).$$ In our case (number of
edges is $p+1$) we can estimate (Theorem 7.1) the degree of field
$L$:
\begin{itemize}
    \item either $se_\tau/(n-1)$ is integer;
    \item or $se_\tau/(n+1)$ is integer;
    \item or $s(e_\tau-2k)/(n-1)$ is integer and $>k$.
\end{itemize}
\begin{rem} In \cite{Zapponi}
$$s=\text{gcd}(\,s_is_j,\, 1\leqslant i<j\leqslant k,\,
s_i(s_i-1),\, 1\leqslant i\leqslant k)$$ and $se_\tau/n$ is
integer. Thus, Zapponi obtained better estimations on degree of
definition field of a plane tree with prime number of edges.
\end{rem}
\begin{rem} We will need the following result from
algebraic number theory (see \cite{Bor}). Let $K$ be an algebraic
field of degree $n$, $p$ --- a prime number,
$\rho_1,\ldots,\rho_m$ --- all prime divisors in $K$, that divide
$p$, i.e.
$$p=c\cdot\rho_1^{e_1}\cdot\ldots\cdot \rho_m^{e_m},$$ where $c$
is a unit and $e_i$ --- ramification indices. Then
$$n=\sum_{i=1}^m e_in_i, \eqno(1.1)$$ where $n_i$ --- inertia
indices.\end{rem} The case of passport $a_1^4a_2^2b_1b_7$ is
considered in Example 7.1. We have five maps with this passport.
Here $n+1=6$ and $s=2$. Thus, either $e_\tau$ is even or $e_\tau$
is divisible by $3$. As degree of definition field is $\leqslant
5$, we have two possibilities: either we have two Galois orbits of
cardinality 2 and 3, or number $7$ is the product of two prime
divisors $7=c\rho_1^2\rho_2^3$ ($c$ is a unit). Actually, the
first possibility holds.

\section{Enumeration of $2^1$ maps}
\pn A $2^1$-map can be obtained by joining a white vertex of some
tree $T_1$ with black vertex of another tree $T_2$ by a double
edge:
\[\begin{picture}(275,50) \put(0,25){\circle*{3}}
\put(20,5){\circle*{3}} \put(20,45){\circle*{3}}
\put(20,25){\circle{4}} \put(0,25){\line(1,0){18}}
\put(20,5){\line(0,1){18}} \put(20,27){\line(0,1){18}}
\put(50,22){$+$} \put(90,25){\circle*{3}}\put(125,10){\circle*{3}}
\put(110,25){\circle{4}} \put(90,25){\line(1,0){18}}
\put(125,10){\circle*{3}} \put(125,40){\circle*{3}}
\put(111,26){\line(1,1){14}} \put(111,24){\line(1,-1){14}}
\put(150,22){$=$}

\put(190,25){\circle*{3}} \put(210,25){\circle{4}}
\put(190,25){\line(1,0){18}} \put(195,10){\circle*{3}}
\put(195,40){\circle*{3}} \put(195,10){\line(1,1){14}}
\put(195,40){\line(1,-1){14}} \put(240,25){\circle*{3}}
\put(260,25){\circle{4}} \put(240,25){\line(1,0){18}}
\put(275,10){\circle*{3}} \put(275,40){\circle*{3}}
\put(261,26){\line(1,1){14}} \put(261,24){\line(1,-1){14}}
\qbezier(211,26)(225,40)(240,25) \qbezier(211,24)(225,10)(240,25)
\end{picture}\] At first we will explain how plane bipartite trees
were enumerated in the work \cite{Goul}. \pmn Let $M$ be the set
of all plane bipartite trees with the given passport
$P=a_1^{k_1}a_2^{k_2}\ldots\\ b_1^{l_1}b_2^{l_2}\ldots$. Here
$n=k_1+k_2+\ldots$ and $m=l_1+l_2+\ldots$ are numbers of white and
black vertices, respectively. Let $\#\text{Aut}(T)$ be the order
of the group of automorphisms of \emph{plane bipartite} tree $T$,
then
$$\sum_{T\in M}\frac{1}{\#\text{Aut}(T)}=
\frac{(n-1)!\,(m-1)!}{\overset{\vspace{1mm}}{\prod_i
(k_i)!\,\prod_j (l_j)!}}\,.$$
\begin{ex} There are two trees with passport
$P=a_1^2a_4b_1^2b_2^2$:
\[\begin{picture}(160,60) \put(0,30){\circle{4}}
\put(15,30){\circle*{3}} \put(35,30){\circle{4}}
\put(35,50){\circle*{3}} \put(35,10){\circle*{3}}
\put(55,30){\circle*{3}} \put(70,30){\circle{4}}
\put(2,30){\line(1,0){31}} \put(35,32){\line(0,1){18}}
\put(35,28){\line(0,-1){18}} \put(37,30){\line(1,0){31}}

\put(90,27){and}

\put(120,45){\circle*{3}} \put(120,15){\circle*{3}}
\put(135,30){\circle{4}} \put(134,31){\line(-1,1){14}}
\put(134,29){\line(-1,-1){14}} \put(150,15){\circle*{3}}
\put(150,45){\circle*{3}} \put(160,5){\circle{4}}
\put(160,55){\circle{4}} \put(136,31){\line(1,1){23}}
\put(136,29){\line(1,-1){23}}
\end{picture}\] The group of automorphisms of the left tree has
order 2, the group of automorphisms of the right tree is trivial.
Thus
$$1+\frac 12=\frac 32= \frac{2!\,3!}{2!\,2!\,2!}\,$$
\end{ex} \pmn
Let $K$ be the set of infinite sequences of nonnegative integers
$k=(k_1,k_2,\ldots)$ such, that only finite number of $k_i$ are
positive. Let $|k|=\sum_{i=1}^\infty k_i$,
$||k||=\sum_{i=1}^\infty i\cdot k_i$, $k!=\prod_{i=1}^\infty
(k_i)!$ and $a^k=\prod_{i=1}^\infty a_i^{k_i}$, where $a_i$ are
formal variables. Let us consider series
$$A(a,t,x)=\sum_{k\in K}\frac{(|k|-1)!\,a^k}{k!}\cdot t^{|k|}
x^{||k||}=\sum_{n=1}^\infty \frac 1n\cdot(t\cdot x\cdot a_1+
t\cdot x^2 \cdot a_2+t\cdot x^3\cdot a_3+\ldots)^n$$ and
analogously defined series
$$B(b,t,y)=\sum_{l\in K} \frac{(|l|-1)!\,b^l}{l!} \cdot t^{|l|} y^{|| l||}=
\sum_{m=1}^\infty \frac 1m\cdot (t\cdot y\cdot b_1+t\cdot y^2\cdot
b_2+ t\cdot y^3\cdot b_3+\ldots)^m.$$ Now we can construct the
generating function $T(a,b,t,x)$ for the set of plane bipartite
trees: from the product $A(a,t,x)\cdot B(b,t,y)$ we delete those
terms, where ${\rm deg}(x)\neq{\rm deg}(y)$ or ${\rm deg}(x)\neq
{\rm deg}(t)-1$ and add terms $a_0$ and $b_0$ that correspond to
trees with one vertex (white or black). \pmn Let $D_a$ and $D_b$
be differential operators:
$$\begin{array}{c}D_a=a_2x\cdot \dfrac{\partial}{\partial
a_0}+a_3x\cdot\dfrac{\partial}{\partial a_1}+2a_4x\cdot
\dfrac{\partial}{\partial a_2}+3a_5x\cdot
\dfrac{\partial}{\partial a_3}+\ldots\\ \\ \text{and}\\ \\
D_b=b_2x\cdot \dfrac{\partial}{\partial
b_0}+b_3x\cdot\dfrac{\partial}{\partial b_1}+2b_4x\cdot
\dfrac{\partial}{\partial b_2}+3b_5x\cdot
\dfrac{\partial}{\partial b_3}+\ldots\end{array}$$
\begin{theor} Let $M(a,b,t,x)$ be the generating function of
$2^1$-maps, then
$$M(a,b,t,x)=D_a(T)\cdot D_b(T).$$ \end{theor}
\begin{proof} We connect a white vertex $v$, ${\rm deg}(v)=k$,
of one tree with a black vertex $u$, ${\rm deg}(u)=l$ of another
by the double edge. There are $kl$ ways to draw this double edge.
This procedure increases the degree of $v$ by 2 and the degree of
$u$ by two. It remains to note that the group of automorphisms of
a $2^1$-map is trivial.\end{proof}
\begin{ex} Terms of $M(a,b,t,x)$ with ${\rm deg}(x)={\rm deg}(t)
=5$ are as follows:
$$\begin{array}{l}
(a_5b_2b_1^3+a_2a_1^3b_5+a_4a_1b_3b_1^2+a_3a_1^2b_4b_1+
2a_4a_1b_2^2b_1+\\ \hspace{1cm}+2a_2^2a_1b_4b_1+
2a_3a_2b_3b_1^2+2a_3a_1^2b_3b_2+
a_3a_2b_2^2b_1+a_2^2a_1b_3b_2)x^5t^5\end{array}$$ We see that
there are fourteen $2^1$-maps with 5 edges. Here they are:
\[\begin{picture}(300,55) \put(0,35){\circle*{3}}
\put(5,20){\circle*{3}} \put(5,50){\circle*{3}}
\put(20,35){\circle{4}} \put(0,35){\line(1,0){18}}
\put(5,20){\line(1,1){14}} \put(5,50){\line(1,-1){14}}
\put(50,35){\circle*{3}} \qbezier(21,36)(35,55)(50,35)
\qbezier(21,34)(35,15)(50,35) \put(15,2){$a_5b_1^3b_2$}

\put(80,35){\circle{4}} \put(110,35){\circle*{3}}
\qbezier(81,36)(95,55)(110,35) \qbezier(81,34)(95,15)(110,35)
\put(125,50){\circle{4}} \put(130,35){\circle{4}}
\put(125,20){\circle{4}} \put(110,35){\line(1,1){14}}
\put(110,35){\line(1,0){18}} \put(110,35){\line(1,-1){14}}
\put(90,2){$a_1^3a_2b_5$}

\put(160,20){\circle*{3}} \put(160,50){\circle*{3}}
\put(175,35){\circle{4}} \put(160,20){\line(1,1){14}}
\put(160,50){\line(1,-1){14}} \put(205,35){\circle*{3}}
\qbezier(176,36)(190,55)(205,35) \qbezier(176,34)(190,15)(205,35)
\put(220,35){\circle{4}} \put(205,35){\line(1,0){13}}
\put(175,2){$a_1a_4b_1^2b_3$}

\put(240,35){\circle*{3}} \put(255,35){\circle{4}}
\put(240,35){\line(1,0){13}} \put(285,35){\circle*{3}}
\qbezier(256,36)(270,55)(285,35) \qbezier(256,34)(270,15)(285,35)
\put(300,20){\circle{4}} \put(300,50){\circle{4}}
\put(285,35){\line(1,1){14}} \put(285,35){\line(1,-1){14}}
\put(250,2){$a_1^2a_3b_1b_4$}
\end{picture}\]

\[\begin{picture}(300,60)
\put(0,10){\circle{4}} \put(10,20){\circle*{3}}
\put(10,50){\circle*{3}} \put(25,35){\circle{4}}
\put(1,11){\line(1,1){23}} \put(10,50){\line(1,-1){14}}
\put(55,35){\circle*{3}} \qbezier(26,37)(40,55)(55,35)
\qbezier(26,34)(40,15)(55,35) \put(15,2){$a_1a_4b_1b_2^2$}

\put(80,60){\circle{4}} \put(90,20){\circle*{3}}
\put(90,50){\circle*{3}} \put(105,35){\circle{4}}
\put(81,59){\line(1,-1){23}} \put(90,20){\line(1,1){14}}
\put(135,35){\circle*{3}} \qbezier(106,37)(120,55)(135,35)
\qbezier(106,34)(120,15)(135,35) \put(95,2){$a_1a_4b_1b_2^2$}

\put(160,35){\circle{4}} \put(190,35){\circle*{3}}
\qbezier(161,37)(175,55)(190,35) \qbezier(161,34)(175,15)(190,35)
\put(204,21){\circle{4}} \put(204,49){\circle{4}}
\put(215,10){\circle*{3}} \put(190,35){\line(1,1){13}}
\put(190,35){\line(1,-1){13}} \put(215,10){\line(-1,1){10}}
\put(165,2){$a_1a_2^2b_1b_4$}

\put(240,35){\circle{4}} \put(270,35){\circle*{3}}
\qbezier(241,37)(255,55)(270,35) \qbezier(241,34)(255,15)(270,35)
\put(284,21){\circle{4}} \put(284,49){\circle{4}}
\put(295,60){\circle*{3}} \put(270,35){\line(1,1){13}}
\put(270,35){\line(1,-1){13}} \put(295,60){\line(-1,-1){10}}
\put(245,2){$a_1a_2^2b_1b_4$}
\end{picture}\]

\[\begin{picture}(255,60)
\put(0,35){\circle*{3}} \put(15,35){\circle{4}}
\put(0,35){\line(1,0){13}} \put(45,35){\circle*{3}}
\qbezier(16,37)(30,55)(45,35) \qbezier(16,34)(30,15)(45,35)
\put(58,35){\circle{4}} \put(70,35){\circle*{3}}
\put(45,35){\line(1,0){11}} \put(70,35){\line(-1,0){10}}
\put(15,2){$a_2a_3b_1^2b_3$}

\put(100,35){\circle{4}} \put(130,35){\circle*{3}}
\qbezier(101,37)(115,55)(130,35) \qbezier(101,34)(115,15)(130,35)
\put(144,35){\circle{4}} \put(130,35){\line(1,0){12}}
\put(155,24){\circle*{3}} \put(155,46){\circle*{3}}
\put(155,24){\line(-1,1){10}} \put(155,46){\line(-1,-1){10}}
\put(110,2){$a_2a_3b_1^2b_3$}

\put(185,35){\circle{4}} \put(210,35){\circle{4}}
\put(197,35){\circle*{3}} \put(187,35){\line(1,0){21}}
\put(240,35){\circle*{3}} \qbezier(211,37)(225,55)(240,35)
\qbezier(211,34)(225,15)(240,35) \put(255,35){\circle{4}}
\put(240,35){\line(1,0){13}} \put(205,2){$a_1^2a_3b_2b_3$}
\end{picture}\]

\[\begin{picture}(255,55)
\put(0,46){\circle{4}} \put(0,24){\circle{4}}
\put(11,35){\circle*{3}} \put(25,35){\circle{4}}
\put(11,35){\line(1,0){12}} \put(11,35){\line(-1,1){10}}
\put(11,35){\line(-1,-1){10}} \put(55,35){\circle*{3}}
\qbezier(26,37)(40,55)(55,35) \qbezier(26,34)(40,15)(55,35)
\put(15,2){$a_1^2a_3b_2b_3$}

\put(95,35){\circle*{3}} \put(105,35){\circle{4}}
\put(130,35){\circle{4}} \put(117,35){\circle*{3}}
\put(95,35){\line(1,0){8}} \put(107,35){\line(1,0){21}}
\put(160,35){\circle*{3}} \qbezier(131,37)(145,55)(160,35)
\qbezier(131,34)(145,15)(160,35) \put(115,2){$a_2a_3b_1b_2^2$}

\put(190,35){\circle{4}} \put(220,35){\circle*{3}}
\qbezier(191,37)(205,55)(220,35) \qbezier(191,34)(205,15)(220,35)
\put(233,35){\circle{4}} \put(245,35){\circle*{3}}
\put(255,35){\circle{4}} \put(220,35){\line(1,0){11}}
\put(235,35){\line(1,0){18}} \put(205,2){$a_1a_2^2b_2b_3$}
\end{picture}\]
\end{ex}

\section{Normalized model}
\pn We will study $2^1$-maps with $p+1$ edges, where $p$ is a
prime number. Let $D$ be such map with $n+1$ white vertices and
$m+1$ black (here $n+m=p-1$). Let $\beta$ be a Belyi function for
$D$ such, that some white vertex of degree $>1$ is in 0 and some
black vertex in 1. Let $v_0,v_1,\ldots,v_n$ be white vertices of
degrees $k_0,k_1,\ldots,k_n$, respectively, and let
$x_0,x_1,\ldots,x_n$ be their coordinates. Analogously, let
$u_0,u_1,\ldots,u_m$ be black vertices of degrees
$l_0,l_1,\ldots,l_m$, respectively, and let $y_0,y_1,\ldots,y_m$
be their coordinates. Function $\beta$ has a pole at point $c$
inside the face of perimeter 2. We can assume that
$$\beta=\dfrac{\prod_{i=0}^n (z-x_i)^{k_i}}{z-c}\,.$$ If
$\beta(y_j)=r$, $j=0,\ldots,m$, then
$$\beta-r=\dfrac{\prod_{j=0}^m (z-y_j)^{l_j}}{z-c}\,.$$ Let $K$ be
the big field of definition of function $\beta$, i.e. $K$ contains
coordinates of all vertices (and numbers $c$ and $r$). Let $\rho$
be a prime divisor in $K$ that divides $p$, $v_\rho$ be the
corresponding valuation and $e$ --- ramification index of $p$ with
respect to $\rho$. Let $O$ be the ring of $\rho$-integral
elements, i.e.
$$O=\{k\in K\,|\,v_\rho(k)\geqslant 0\}$$ and let $I=\{k\in
K\,|\, v_\rho(k)>0\}$ be the maximal ideal in $O$. \pmn If
$$w=\min_{0\leqslant j\leqslant m} v_\rho(y_j) \text{ and }
v_\rho(y_i)=w,$$ then
$w\leqslant 0$, because some black vertex is at 1. \pmn Now let us
make a change of variables: all coordinates we divide by $y_i$. We
preserve the notation, but will remember that coordinates of all
black vertices are $\rho$-integral (and some black vertex is at
1). \pmn In equation
$$\prod_{i=0}^n (z-x_i)^{k_i}-r(z-c)=\prod_{j=0}^m (z-y_j)^{l_j}
\eqno(3.1)$$ the right-hand side is $\rho$-integral, hence, the
left-hand side is also $\rho$-integral. But some white vertex of
degree $>1$ is at 0, so
$$\prod_{i=0}^n (z-x_i)^{k_i}=z^{p+1}+a_1z^p+a_2z^{p-1}+\ldots+
a_{p-1}z^2$$ and coefficients $a_1,\ldots,a_{p-1}$ are
$\rho$-integral. Hence, coordinates $x_0,\ldots,x_n$ are also
$\rho$-integral. \pmn Let us consider the numerator of derivative
$\beta'$:
\[\parbox{10cm}
{$\begin{array}{l}((p+1)z^p+p\,a_1z^{p-1}+(p-1)a_2
z^{p-2}+\ldots+2a_{p-1}z)(z-c)-\\ \\ \hspace{1cm}-
(z^{p+1}+a_1z^p+a_2z^{p-1}+\ldots+a_{p-1}z^2)=
\\ \\ \hspace{2cm}
=p\,\prod_{i=0}^n (z-a_i)^{k_i-1} \prod_{j=0}^m
(z-y_j)^{l_j-1}.\end{array}$} \hspace{2cm} (3.2)\] Coefficients of
the polynomial in the right-hand side of (3.2) belong to $I$,
hence, the same is true about the left side of (3.2). So, we have
the following relations:
\[\hspace{2cm}\parbox{4cm}{$\begin{array}{l}
-(p+1)c+(p-1)\,a_1\in I\\
-p\;c\,a_1+(p-2)\,a_2
\in I\\ -(p-1)\,c\,a_2+(p-3)\,a_3\in I\\
\qquad\ldots\qquad\ldots\qquad\ldots\\ -3\,c\,a_{p-2}+a_{p-1}\in I
\end{array}$} \hspace{6cm}(3.3)\] \pmn
The first relation gives us the $\rho$-integrality of $c$. The
second --- that $a_2\in I$. The third --- that $a_3\in I$, and so
on. We see that all coefficients of the polynomial
$z^{p+1}+a_1z^p+\ldots+a_{p-1}z^2$ except, maybe $a_1$, belong to
$I$. It means that coordinates of all white vertices except, maybe
one of degree 1, belong to $I$. \pmn Now let us consider black
vertices.
\begin{prop} There exists a black vertex of degree $>1$ such, that
its coordinate does not belong to $I$. \end{prop} \pmn The proof
of this proposition will be given in Section 4. \pmn The
coordinate $y_s$ of some black vertex $u_s$ of degree $>1$ does
not belong to ideal $I$, i.e. $v_\rho(y_s)=0$. Let us make a
change of variable: all coordinates we divide by $y_s$. As before
some white vertex of degree $>1$ is at 0 and coordinates of all
white vertices except, maybe one of degree 1, belong to $I$. But
now some black vertex of degree $>1$ is at 1. \pmn Now the change
$z:=1-z$ and $\beta:=\beta-r$ allows one to apply  the above
reasoning to black vertices also.
\begin{defin} Belyi function $\beta$ of some $2^1$-map $D$ with
$p+1$ edges is called its \emph{normalized model} (with respect to
a prime divisor $\rho$), if:
\begin{itemize}
    \item coordinates of all vertices (black and white) and
    coordinate of the pole are $\rho$-integral
    \item some white vertex of degree $>1$ is at $0$ and some
    black vertex of degree $>1$ is at $1$;
    \item coordinates of all white vertices except, maybe one of
    degree 1, belong to the ideal $I$;
    \item coordinates of all black vertices except, maybe one of
    degree 1, belong to the set $\{1+I\}$.
\end{itemize} $\beta$ is called \emph{completely normalized}, if
its critical values are $0$ (value in white vertices) and $1$
(value in black vertices).
\end{defin}
\begin{theor} For each $2^1$-map with $p+1$ edges there exists a
normalized model. \end{theor}

\section{The proof of Proposition 3.1.}
\begin{proof} Let us assume that there exists a $2^1$-map $D$ and
its Belyi function $\beta$ such, that
\begin{itemize}
    \item coordinates of all vertices are $\rho$-integral and the
    same is true for $c$;
    \item a white vertex of degree $>1$ is at 0 and coordinates of
    all other white vertices (except, maybe one of degree 1)
    belong to ideal $I$;
    \item some black vertex is at 1 and coordinates of all
    black vertices of degree $>1$ belong to $I$.
\end{itemize} As coefficients of polynomial
$$\prod_{j=0}^m (z-y_j)^{l_j}=z^{p+1}+b_1z^p+\ldots+b_{p-1}z^2+
b_pz+b_{p+1}$$ are equal to the corresponding coefficients $a_i$
of polynomial $\prod_i (z-x_i)^{k_i}$ for $i<p$, then $b_i\in I$
for $i>1$ and coordinates of all black vertices (except one that
is at 1) belong to $I$. \pmn As $a_1=b_1$, there exists a white
vertex of degree 1 whose coordinate belongs to the set $\{1+I\}$.
Let $y_0=1$, $y_j\in I$ for $j>0$, $x_0-1\in I$ and $x_i\in I$ for
$i>0$. As $c+a_1\in I$ (see (3.3)), then $c-1\in I$. \pmn Let
$$k=\min\bigl(\min_{i>0} v_\rho(x_i),\min_{j>0} v_\rho(y_j)
\bigr)>0.$$ We can assume that $v_\rho(x_s)=k$, $s>0$. Let us
consider the function
$$\beta_1=x_s^{-p}\beta(x_s z)=x_s^{-p}\cdot\frac{\prod_{i=0}^n
(x_s z-x_i)^{k_i}}{x_s z-c}=\frac{(x_s
z-x_0)\prod_{i>0}(z-\widetilde{x}_i)^{k_i}}{x_s z-c}\,,$$ where
$\widetilde{x}_i=x_i/x_s$. Analogously, let $\widetilde{y}_j=
y_j/x_s$. \pmn Let us find the derivative $\beta'_1$:
\begin{multline*}\beta'_1=x_s^{-p}x_s\beta'(x_s z)=
x_s^{-p}x_s\, \frac{p\prod_{i=0}^n(x_s
z-x_i)^{k_i-1}\prod_{j=0}^m(x_s z-y_j)^{l_j-1}}{(x_s
z-c)^2}=\\=x_s^2\,\frac{p\prod_{i=0}^n(z-\widetilde{x}_i)^{k_i-1}
\prod_{j=0}^m(z-\widetilde{y}_j)^{l_j-1}}{(x_s z-c)^2}
\end{multline*} (because $n+m=p-1$). We see that all coefficients
of the numerator belong to $I$. On the other hand, this numerator
is equal to
$$\left[(x_s z-x_0)\prod_{i=1}^n(z-\widetilde{x}_i)^{k_i}\right]'
(x_s z-c)-x_s\,(x_s z-x_0)\prod_{i=1}^n(z-\widetilde{x}_i)^{k_i}.
\eqno(4.1)$$ As $x_0$, $\widetilde{x}_s$ and $c$ do not belong to
$I$, but $x_s\in I$, then not all coefficients of polynomial (4.1)
belong to the ideal $I$. We have a contradiction.
\end{proof}

\section{Arithmetic of normalized model}
\pn Let $\beta$ be a normalized model.
\subsection{The first case: the coordinate of some white vertex
of degree 1 does not belong to $I$} Then $c\notin I$ and
coordinate of some black vertex of degree 1 does not belong to the
set $\{1+I\}$. Let us assume that $x_0\notin I$ and $y_0-1\notin
I$. \pmn On one hand
$${\rm numerator}(\beta')=\prod_{i=0}^n (z-x_i)^{k_i}\left(\sum_{i=0}^n
\frac{k_i}{z-x_i}\right)(z-c)-\prod_{i=0}^n (z-x_i)^{k_i}.$$ On
the other hand
$${\rm numerator}(\beta')=p\prod_{i=0}^n (z-x_i)^{k_i-1}\prod_{j=0}^m
(z-y_j)^{l_j-1}.$$ Thus,
$$\prod_{i=0}^n (z-x_i)\cdot\left(\sum_{i=0}^n \frac{k_i}{z-x_i}\right)
(z-c)-\prod_{i=0}^n (z-x_i)=p\prod_{j=0}^m(z-y_j)^{l_j-1}\,.
\eqno(5.1)$$ Let the vertex $v_1$ be at zero, i.e. $x_1=0$. Then
$$k_1x_0c\prod_{i>1}x_i=\pm p\prod_{j=0}^m y_j^{l_j-1} \Rightarrow
\sum _{i=2}^n v_\rho(x_i)=e \Rightarrow (n-1)\,k\leqslant e,
\eqno(5.2)$$ where $k=\min_{i>1}v_\rho(x_i)$. \pmn Now, as in
previous section, we will work with the function
$$\beta_1=x_s^{-p}\beta(x_s z)=x_s^{-p}\,\frac{(x_s z-x_0)
\prod_{i=1}^n (x_s z-x_i)^{k_i}}{x_s z-c}=\frac{(x_s z-x_0)
\prod_{i=1}^n (z-\widetilde{x}_i)^{k_i}}{x_s z-c}\,,$$ where
$v_\rho(x_s)=k$ and $\widetilde{x}_i=x_i/x_s$. \pmn As in the
previous section we deduce, that not all coefficients of the
numerator of derivative
\begin{multline*}\beta'_1(z)=x_s^{1-p}\beta'(x_s
z)=x_s^{1-p}\, \frac{p\,\prod_{i=1}^n(x_s z-x_i)^{k_i-1}
\prod_{j=0}^m(x_s z-y_j)^{l_j-1}}{(x_s z-c)^2}=\\=x_s^{1-n}\,
\frac{p\,\prod_{i=1}^n (z-\widetilde{x}_i)^{k_i-1}
\prod_{j=0}^m(x_s z-y_j)^{l_j-1}}{(x_s z-c)^2}\end{multline*}
belong to $I$. Thus, $(n-1)\,k\geqslant e$. Taking (5.2) into
account we have that
$$v_\rho(x_i)=\frac{e}{n-1}\, \text{ for $i>0$ and $x_i\neq 0$.}
\eqno(5.3)$$ The substitution $x_j$ instead of $z$ in (5.1), where
$j>0$, $x_j\neq 0$, and subsequent computing  valuations gives the
equality
$$\sum_{i>0,i\neq j} v_\rho(x_i-x_j)=e.$$ As
$v_\rho(x_i-x_j)\geqslant e/(n-1)$, then
$$v_\rho(x_i-x_j)=\frac{e}{n-1}\, \text{ for $i,j>0$ and $i\neq j$.}
\eqno(5.4)$$

\subsection{The second case: all $x_i\in I$} In this case
$c\in I$ and coordinates of all black vertices belong to the set
$\{1+I\}$, except one of degree 1 --- coordinate of this vertex
belong to $I$. Let $k=\min v_\rho(x_i)$ and let
$$J=\{i, 0\leqslant i\leqslant n\,|\, v_\rho(x_i)=k\}.$$

\subsubsection{Either the cardinality of $J$ is $>1$, or $J=\{l\}$
and degree of $v_l$ is $>1$}  Let $x_0=0$. After the substitution
$z=0$ in (5.1) and computing valuations we get the equality
$$\sum_{i=1}^n v_\rho(x_i)+v_\rho(c)=e. \eqno(5.5)$$ Thus,
$v_\rho(c)<e$ and $v_\rho(x_i)<e$ for $i=1,\ldots,n$. As
$$v_\rho(-(p+1)c+(p-1)a_1)\geqslant e$$ (see (3.2)), then
$v_\rho(c)\geqslant k$ and from (5.5) we get the relation
$(n+1)k\leqslant e$. \pmn We will work with function
$$\beta_1=x_s^{-p}\beta(x_s z)= \frac{\prod_{i=0}^n
(z-\widetilde{x}_i)^{k_i}}{z-\widetilde{c}}\,,$$ where, as above,
$v_\rho(x_s)=k$, $\widetilde{x}_i=x_i/x_s$ and $\widetilde{c}=
c/x_s$. Let
$$\prod_{i=0}^n (z-\widetilde{x}_i)^{k_i}=z^{p+1}+\widetilde{a}_1
z^p+\widetilde{a}_2z^{p-1}+\ldots$$ Then
$$\beta'_1=\frac{(z-\widetilde{c})\cdot\bigl[(p+1)z^p+
p\,\widetilde{a}_1z^{p-1}+\ldots \bigr]-
\bigl[z^{p+1}+\widetilde{a}_1z^p+\ldots
\bigr]}{(z-\widetilde{c})^2}\,.$$ There exists $j>1$ such, that
$v_\rho(\widetilde{a}_j)=0$, but $v_\rho(\widetilde{a}_i)>0$ for
$i>j$. Now
\begin{multline*}\text{numerator}(\beta'_1)=\ldots+
\bigl[-(p-j+2)\,\widetilde{a}_{j-1}\widetilde{c}+(p-j)\,
\widetilde{a}_j\bigr]\,z^{p-j+1}+\\+\bigl[-(p-j+1)\,\widetilde{a}_j
\widetilde{c}+(p-j-1)\,\widetilde{a}_{j+1}\bigr]\,z^{p-j}\ldots
\end{multline*} If $\widetilde{c}\in I$, then coefficient at
$z^{p-j+1}$ does not belong to $I$. If $\widetilde{c}\notin I$,
then coefficient at $z^{p-j}$ does not belong to $I$.\pmn On the
other hand
$$\beta'_1=x_s^{-(n+1)}\,\frac{p\prod_{i=0}^n(z-\widetilde{x}_i)^{k_i-1}
\prod_{j=0}^m(x_s z-y_j)^{l_j-1}}{(z-\widetilde{c})^2}\,.$$ As not
all coefficients of numerator belong to $I$, then $e\leqslant
(n+1)k$. Hence,
$$v_\rho(x_1)=\ldots=v_\rho(x_n)=v_\rho(c)=\frac{e}{n+1} \text{ and }
v_\rho(x_i-x_j)=\frac{e}{n+1} \text{ for } i\neq j\,. \eqno(5.6)$$

\subsubsection{The degree of $v_0$ is 1, $v_\rho(x_0)=k$ and
$v_\rho(x_i)>k$, if $i>0$} Then $v_\rho(c)=k$. Let $\min_{i>0}
v_\rho(x_i)=l>k$. As
$$\sum_{x_i\neq 0} v_\rho(x_i) +v_\rho(c)=e\,,\text{ then }
2k+(n-1)\,l\leqslant e\,.$$ The same computations with the
function $\beta_1=x_s^{-p} \beta(x_s z)$, where $v_\rho(x_s)=l$,
give us the relation $2k+(n-1)\,l \geqslant e$, i.e.
$$v_\rho(x_i)=\frac{e-2k}{n-1} \text{ for $i>0$ and $x_i\neq 0$
and } v_\rho(x_i-x_j)= \frac{e-2k}{n-1} \text{ for $i,j>0$}\,.
\eqno(5.7)$$  Let us summarize.
\begin{theor} Let $\beta$ be a normalized model of some $2^1$-map
with $p+1$ edges, $n+1$ white vertices and $m+1$ black ones. Let
$\rho$ be a prime divisor in the big definition field $K$, that
divides $p$, $v_\rho$ be the corresponding valuation and
$$I=\{k\in K\,|\,v_\rho(k)>0\}\subset K.$$ Let $e=v_\rho(p)$ be
the ramification index. There are three options.
\begin{itemize}
\item $v_\rho(x_i)=e/(n-1)$ for all white vertices, except one
$v_s$ of degree 1 and $x_s\notin I$. Moreover,
$v_\rho(x_i-x_j)=e/(n-1)$ for $i,j\neq s$. \item
$v_\rho(x_i)=e/(n+1)$ for all $i$, $v_\rho(x_i-x_j)=e/(n+1)$ for
$i,j$. \item $x_i\in I$ for all $I$ and $v_\rho(x_i)=l$ for all
white vertices, except one $v_s$ of degree 1 for which
$v_\rho(x_s)=k<l$. Moreover $l=(e-2k)/(n-1)$ and
$v_\rho(x_i-x_j)=l$ for all $i,j\neq s$.
\end{itemize} \end{theor}
\begin{rem} Analogous results are valid for black vertices. \end{rem}

\section{Canonical model}
\pn  \begin{defin} A Belyi function $\beta$ of some $2^1$-map is
called its canonical model if
\begin{itemize}
    \item value of $\beta$ is 0 in white vertices and 1 --- in
    black;
    \item sum of coordinates of white vertices of degree $>1$ is
    0;
    \item sum of coordinates of black vertices of degree $>1$ is
    1.
\end{itemize} \end{defin}
\begin{theor} For each $2^1$-map with $p+1$ edges there exists the
unique canonical model. \end{theor}
\begin{proof} Let $\beta$ be a normalized model of our map with
the big definition field $K$, $X$ be the sum of coordinates of
white vertices of degree $>1$ and $Y$ be the sum of black vertices
of degree $>1$. Then $X\in I$, but $Y\notin I$. Indeed, otherwise
the number of black vertices of degree $>1$ is $p$ and the number
of edges is not less, than $2p$. \pmn For a new variable $z'=az+b$
these sums are $X'=aX+bk$ and $Y'=aY+bl$, where $k$ is the number
of white vertices of degree $>1$ and $l$ is the corresponding
number of black vertices. We have to find numbers $a$ and $b$
such, that
$$\left\{\begin{array}{l} aX+bk=0\\ aY+bl=1\end{array}\right.$$
As $X\,l-Y\,k$ --- the determinant of the system, is $\notin I$
(and, thus, $\neq 0$), then the system has the unique solution and
$$a=-\frac{k}{X\,l-Y\,k},\quad b=\frac{X}{X\,l-Y\,k}.\,$$ We see,
that $a$ is $\rho$-integral, but $a\notin I$. Let
$\widetilde{x}_i=ax_i+b$, $i=0,\ldots,n$ $\widetilde{y}_j=ay_j+b$,
$j=0,\ldots,m$, and $\widetilde{c}=ac+b$. As
$\widetilde{x}_i-\widetilde{x}_j= a(x_i-x_j)$, then
$$v_\rho(\widetilde{x}_i-\widetilde{x}_j)=v_\rho(x_i-x_j).$$
Let $L$ be the field of definition of polynomials
$$P=\prod_{i=0}^n(z-\widetilde{x}_i)^{k_i} \text{ and }
Q=\prod_{j=0}^m (z-\widetilde{y}_j)^{l_j}.$$ As
$P-r(z-\widetilde{c})=Q$, then $r\in L$ and $\widetilde{c}\in L$.
Hence, the function $\widetilde{\beta}=P/r(z-\widetilde{c})$ will
be the canonical model. \end{proof}
\begin{rem} The uniqueness of canonical model means that its field
of definition coincides with the definition field of the map
\end{rem}

\section{Geometrical ramification}
\pn Let $D$ be a $2^1$-map, $k$ --- the maximal degree of its
white vertices and $s_i$, $i=2,\ldots,k$, --- the number of white
vertices of degree $i$. Let
$$s=\text{gcd}(\,s_is_j,\, 2\leqslant i<j\leqslant k,\,
s_i(s_i-1),\, 2\leqslant i\leqslant k).$$ There exists integers
$c_i$ and $c_{ij}$ such, that
$$S=\sum_i c_is_i(s_i-1)+\sum_{i<j} c_{ij}s_is_j.$$ Let $\beta$ be
the canonical model of a map $D$, $v_0,\ldots,v_n$ --- white
vertices, $x_0,\ldots,x_n$ --- their coordinates and
$d_0,\ldots,d_n$ --- their degrees. Function $\beta$ is defined
over the field $L$. Elements
$$t_i=\prod_{i_1\neq i_2,\,d_{i_1}=d_{i_2}=i}
(x_{i_1}-x_{i_2}),\,i>1, \text{ and }
t_{ij}=\prod_{i_1,\,i_2,\,d_{i_1}=i,\,d_{i_2}=j}
(x_{i_1}-x_{i_2}),\,1<i<j$$ belong to $L$. \pmn Let $\tau$ be a
prime divisor of the field $L$ that divides $p$ and is divided by
$\rho$, $v_\tau$ be the corresponding valuation and $e_\tau$ be
the ramification index in the field $L$: $e_\tau=v_\tau(p)$. For
each $u\in L$ we have that $v_\rho(u)=v_\rho(\tau)v_\tau(u)$,
thus, $e=v_\rho(\tau)e_\tau$ and
$$\frac{v_\rho(g)}{v_\rho(h)}=\frac{v_\tau(g)}{v_\tau(h)},\,
\text{ for all $g,h\in L$}.\eqno(7.1)$$ Let us consider the
element
$$t=\prod_i t_i^{c_i}\prod_{i<j} t_{ij}^{c_{ij}}\in L.$$ According
to Theorem 5.1,
$$v_\rho(t)=\begin{cases}\text{either }\, se/(n-1),\\ \text{or}
\hspace{7mm} se/(n+1),\\ \text{or} \hspace{7mm}
s(e-2k)/(n-1)>k.\end{cases}$$ Thus, we have the following
statement.
\begin{theor}\pn
\begin{itemize}
    \item Either $se_\tau/(n-1)$ is integer;
    \item or $se_\tau/(n+1)$ is integer;
    \item or $s(e_\tau-2k)/(n-1)$ is integer and $>k$.
\end{itemize} The same can be said about black vertices.
\end{theor} \pmn
\begin{ex} Let us consider $2^1$-maps with passport
$a_1^4a_2^2b_1b_7$. The number of edges is $8=7+1$. There are 5
such maps: \pmn to the right side of the map \,\,\parbox{23mm}{
\begin{picture}(350,40) \put(0,20){\circle{4}} \put(30,20){\circle*{4}}
\put(15,22){\oval(30,28)[t]} \put(15,18){\oval(30,28)[b]}
\put(30,20){\line(1,1){14}} \put(30,20){\line(5,2){18}}
\put(30,20){\line(1,0){21}} \put(30,20){\line(5,-2){18}}
\put(30,20){\line(1,-1){14}} \put(45,35){\circle{4}}
\put(50,28){\circle{4}} \put(53,20){\circle{4}}
\put(50,12){\circle{4}} \put(45,5){\circle{4}}
\end{picture}}  we must add a "tail"\,\, \begin{picture}(20,4)
\put(2,2){\circle{4}} \put(18,2){\circle*{4}}
\put(4,2){\line(1,0){15}}
\end{picture}\,\,. \pmn Here $n+1=6$ and $s=2$, hence, either
$e_\tau$ is even, or $e_\tau$ is divisible by $3$. Thus (see
(1.1)), either we have two Galois orbits (one of cardinality $2$
and another of cardinality $3$), or $7$ is a product of two prime
divisors $\rho_1$ and $\rho_2$ with ramifications $2$ and $3$,
respectively, i.e. $7=\rho_1^2\rho_2^3$. (The first case is
realized.) \end{ex}
\begin{ex} Let us consider $2^1$-maps with passport
$a_1^3a_2a_3b_1^2b_6$. The number of edges is $8=7+1$. There are 8
such maps:
\[\begin{picture}(350,50) \put(0,25){\circle*{3}}
\put(15,25){\circle{4}} \put(50,25){\circle*{3}}
\put(0,25){\line(1,0){13}} \qbezier(16,26)(33,50)(49,26)
\qbezier(16,24)(33,0)(49,24) \put(50,38){\circle{4}}
\put(50,47){\circle*{3}} \put(50,25){\line(0,1){11}}
\put(50,40){\line(0,1){7}} \put(50,12){\circle{4}}
\put(50,25){\line(0,-1){11}} \put(65,35){\circle{4}}
\put(50,25){\line(3,2){14}} \put(65,15){\circle{4}}
\put(50,25){\line(3,-2){14}}

\put(85,25){\circle*{3}} \put(100,25){\circle{4}}
\put(135,25){\circle*{3}} \put(85,25){\line(1,0){13}}
\qbezier(101,26)(118,50)(134,26) \qbezier(101,24)(118,0)(134,24)
\put(135,38){\circle{4}} \put(135,3){\circle*{3}}
\put(135,25){\line(0,1){11}} \put(135,10){\line(0,-1){7}}
\put(135,12){\circle{4}} \put(135,25){\line(0,-1){11}}
\put(150,35){\circle{4}} \put(135,25){\line(3,2){14}}
\put(150,15){\circle{4}} \put(135,25){\line(3,-2){14}}

\put(170,25){\circle*{3}} \put(185,25){\circle{4}}
\put(220,25){\circle*{3}} \put(170,25){\line(1,0){13}}
\qbezier(186,26)(203,50)(219,26) \qbezier(186,24)(203,0)(219,24)
\put(220,38){\circle{4}} \put(220,25){\line(0,1){11}}
\put(220,12){\circle{4}} \put(220,25){\line(0,-1){11}}
\put(235,35){\circle{4}} \put(220,25){\line(3,2){14}}
\put(247,43){\circle*{3}} \put(236,35){\line(3,2){11}}
\put(235,15){\circle{4}} \put(220,25){\line(3,-2){14}}

\put(270,25){\circle*{3}} \put(285,25){\circle{4}}
\put(320,25){\circle*{3}} \put(270,25){\line(1,0){13}}
\qbezier(286,26)(303,50)(319,26) \qbezier(286,24)(303,0)(319,24)
\put(320,38){\circle{4}} \put(320,25){\line(0,1){11}}
\put(320,12){\circle{4}} \put(320,25){\line(0,-1){11}}
\put(335,35){\circle{4}} \put(320,25){\line(3,2){14}}
\put(335,15){\circle{4}} \put(320,25){\line(3,-2){14}}
\put(347,7){\circle*{3}} \put(336,15){\line(3,-2){11}}
\end{picture}\]
\[\begin{picture}(310,60) \put(0,30){\circle{4}}
\put(35,30){\circle*{3}} \qbezier(1,31)(18,55)(34,31)
\qbezier(1,29)(18,5)(34,29) \put(35,43){\circle{4}}
\put(35,30){\line(0,1){11}} \put(23,55){\circle*{3}}
\put(34,44){\line(-1,1){11}} \put(47,55){\circle*{3}}
\put(36,44){\line(1,1){11}} \put(50,40){\circle{4}}
\put(50,20){\circle{4}} \put(35,17){\circle{4}}
\put(35,30){\line(3,2){14}} \put(35,30){\line(3,-2){14}}
\put(35,30){\line(0,-1){11}}

\put(80,30){\circle{4}} \put(115,30){\circle*{3}}
\qbezier(81,31)(98,55)(114,31) \qbezier(81,29)(98,5)(114,29)
\put(115,43){\circle{4}} \put(115,30){\line(0,1){11}}
\put(130,40){\circle{4}} \put(130,20){\circle{4}}
\put(115,17){\circle{4}} \put(115,30){\line(3,2){14}}
\put(115,30){\line(3,-2){14}} \put(115,30){\line(0,-1){11}}
\put(103,5){\circle*{3}} \put(114,16){\line(-1,-1){11}}
\put(127,5){\circle*{3}} \put(116,16){\line(1,-1){11}}

\put(160,30){\circle{4}} \put(195,30){\circle*{3}}
\qbezier(161,31)(178,55)(194,31) \qbezier(161,29)(178,5)(194,29)
\put(195,43){\circle{4}} \put(195,30){\line(0,1){11}}
\put(210,40){\circle{4}} \put(210,20){\circle{4}}
\put(195,17){\circle{4}} \put(195,30){\line(3,2){14}}
\put(195,30){\line(3,-2){14}} \put(195,30){\line(0,-1){11}}
\put(210,50){\circle*{3}} \put(210,42){\line(0,1){8}}
\put(220,40){\circle*{3}} \put(212,40){\line(1,0){8}}

\put(250,30){\circle{4}} \put(285,30){\circle*{3}}
\qbezier(251,31)(268,55)(284,31) \qbezier(251,29)(268,5)(284,29)
\put(285,43){\circle{4}} \put(285,30){\line(0,1){11}}
\put(300,40){\circle{4}} \put(300,20){\circle{4}}
\put(285,17){\circle{4}} \put(285,30){\line(3,2){14}}
\put(285,30){\line(3,-2){14}} \put(285,30){\line(0,-1){11}}
\put(300,10){\circle*{3}} \put(300,18){\line(0,-1){8}}
\put(310,20){\circle*{3}} \put(302,20){\line(1,0){8}}
\end{picture}\] Here $n+1=5$ and $s=1$, hence,  either
$e_\tau$ is divisible by 3, or --- by $5$. Thus (see (1.1)),
either we have two Galois orbits (one of cardinality $3$ and
another of cardinality $5$), or $7$ is a product of two prime
divisors $\rho_1$ and $\rho_2$ with ramifications $3$ and $5$,
respectively, i.e. $7=\rho_1^3\rho_2^5$. (The second case is
realized.) \pmn In normalized model, where white vertex of degree
3 is at zero, black vertex of degree 6 is at 1 and white vertex of
degree 2 is at point $a$ we have the following equation on $a$:
$$15a^8-7\cdot 24a^7+7\cdot 121a^6-7\cdot 360a^5+7^2\cdot 99a^4
-7\cdot 880a^3+7^2\cdot 102a^2-7^2\cdot 48a+7^2=0.$$ Let
$x=v_\rho(a)$, then in this expression there are 3 terms that may
have the minimal valuation: $15a^8$ with valuation $8x$, $7\cdot
880a^3$ with valuation $3x+e$ and $49$ with valuation $2e$. At
least two terms must have the minimal valuation.
\begin{itemize}
    \item $8x=3x+e \text{ minimal } \Rightarrow x=e/5,\,
    2e>3x+e,\,e=5$.
    \item $8x=2e \text{ minimal } \Rightarrow x=e/4,\, 3x+e<2e$,
    contradiction.
    \item $3x+e=2e \text{ minimal }\Rightarrow x=e/3,\, 8x>2e,\,
    e=3$.
\end{itemize} Thus, it again follows that $7=\rho_1^3\rho_2^5$.
\end{ex}
\pbn

\hspace{5mm}
\begin{center}{\Large Supplement}\end{center}
\pbn The difference between normalized and canonical models can be
illustrated on example of their Julia sets (see \cite{Falc} or
\cite{Peit}). Let us consider two maps
\[\begin{picture}(210,55) \put(10,10){\circle*{3}}
\put(30,30){\circle{4}} \qbezier(10,10)(20,0)(30,10)
\qbezier(30,10)(40,20)(31,29) \qbezier(10,10)(0,20)(10,30)
\qbezier(10,30)(20,40)(29,31) \put(30,50){\circle*{3}}
\put(30,32){\line(0,1){18}} \put(60,30){\circle*{3}}
\put(32,30){\line(1,0){28}} \put(75,15){\circle{4}}
\put(75,45){\circle{4}} \put(60,30){\line(1,1){14}}
\put(60,30){\line(1,-1){14}}

\put(110,30){and}

\put(145,5){\circle{4}} \put(160,20){\circle*{3}}
\put(160,50){\circle*{3}} \put(175,35){\circle{4}}
\put(146,6){\line(1,1){28}} \put(160,50){\line(1,-1){14}}
\put(195,35){\circle*{4}} \put(185,37){\oval(20,20)[t]}
\put(185,33){\oval(20,20)[b]} \put(210,35){\circle{4}}
\put(195,35){\line(1,0){13}}\end{picture}\] with passport
$a_1^2a_4b_1b_2b_3$. \pmn In normalized model white vertex of
degree 4 is at $0$, black vertex of degree 3 is at $1$ and 0 and 1
are critical values. Then here are two attracting fixed points ---
$0$ and $1$. In canonical model there is only one attracting fixed
point --- $0$. Let $O$ be the union of small neighborhoods of
attracting fixed points and a neighborhood of infinity. Points
that come into $O$ in $\leqslant 5$ steps are white, in $6$ or $7$
steps --- yellow, in $8$ or $9$ steps --- red. Other points (among
them points of Julia set) are blue. \pbn
\begin{center}{\Large The first map}\end{center}
\pbn
$$\begin{array}{ccc}
\includegraphics[width=6cm,height=4cm]{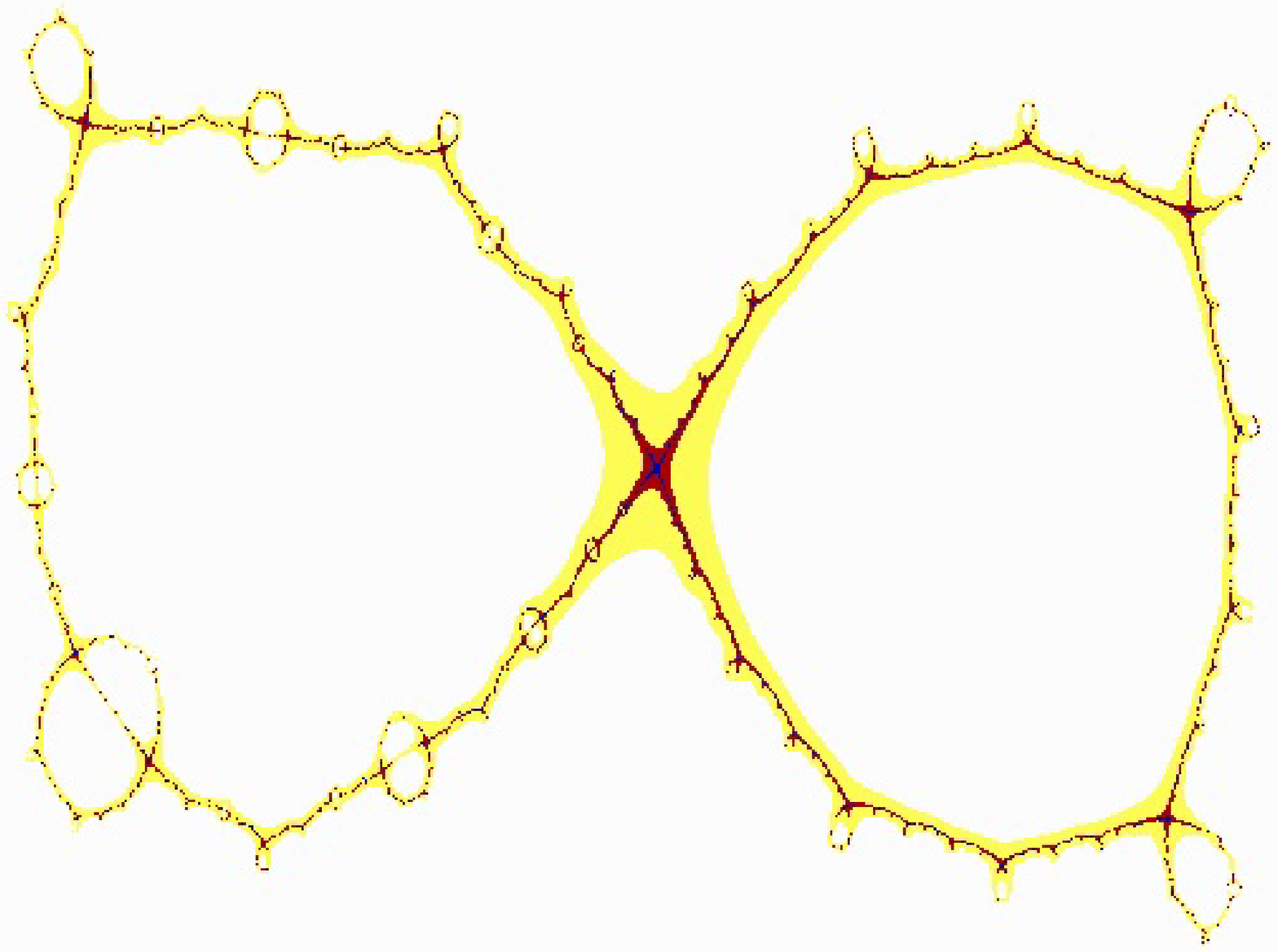} &&
\includegraphics[width=6cm,height=4cm]{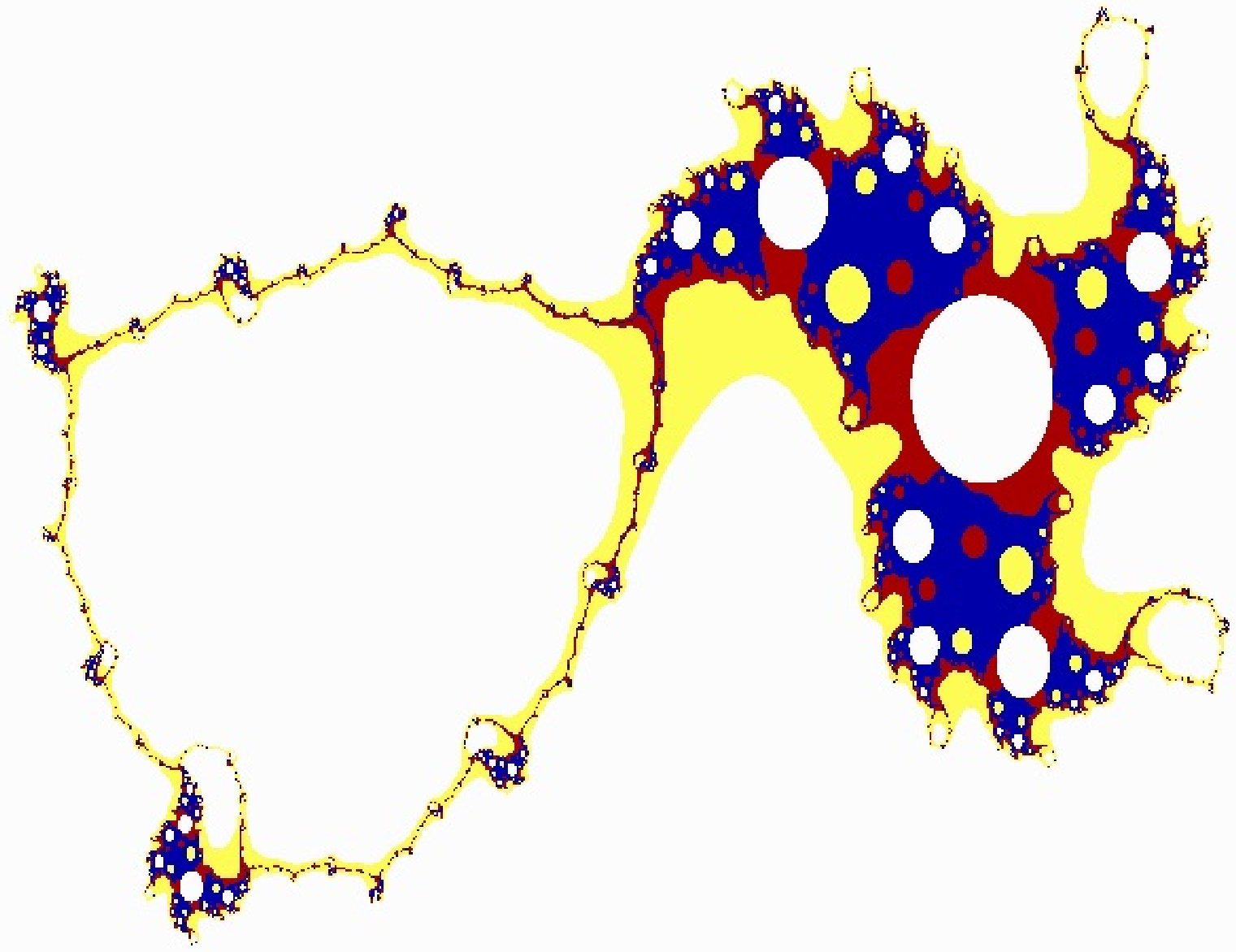}\\ \\
\text{Normalized model}&& \text{Canonical model}
\end{array}$$ \pbn
\begin{center}{\Large The second map}\end{center}
\pbn
$$\begin{array}{ccc}
\includegraphics[width=6cm,height=4cm]{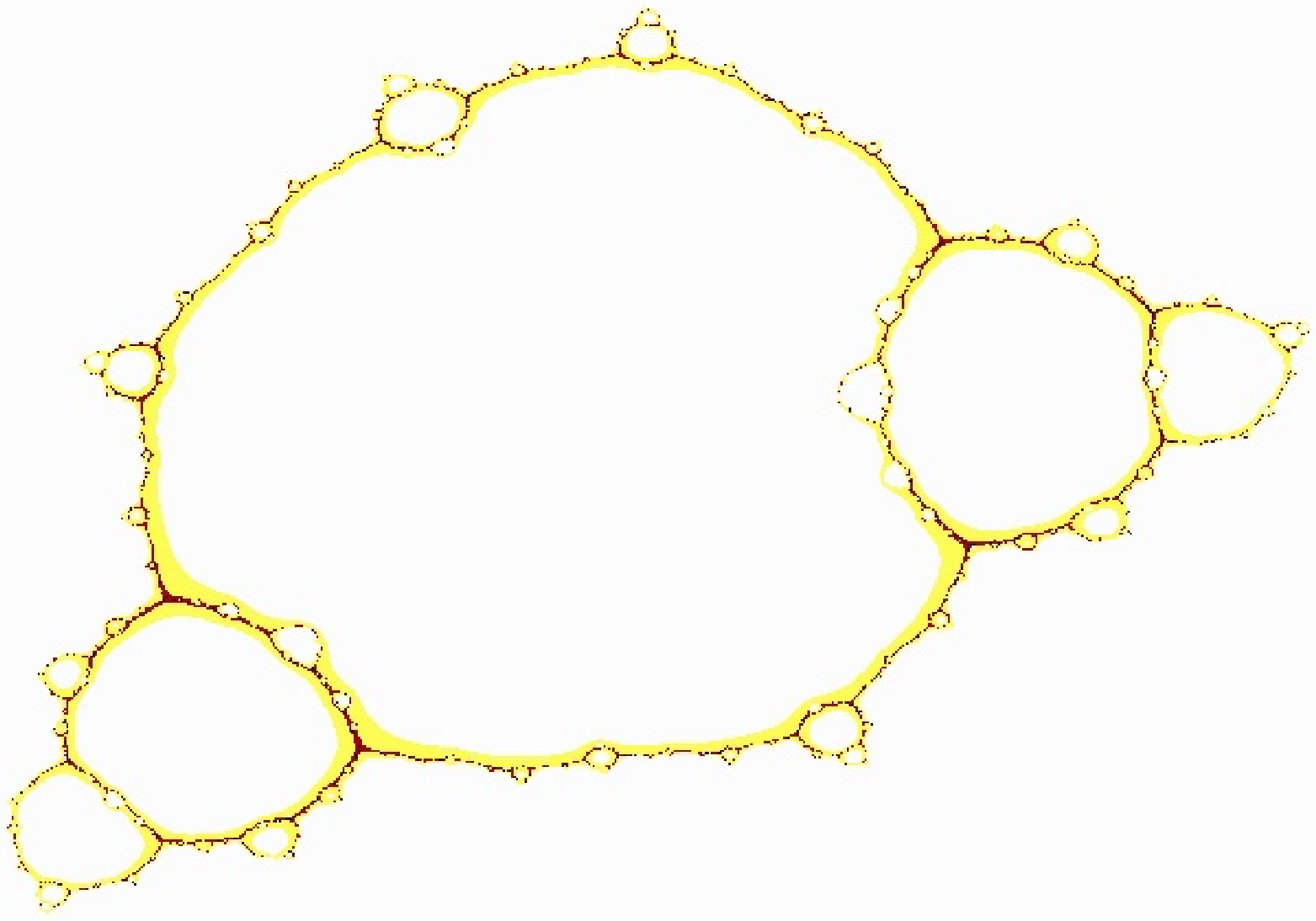} &&
\includegraphics[width=6cm,height=4cm]{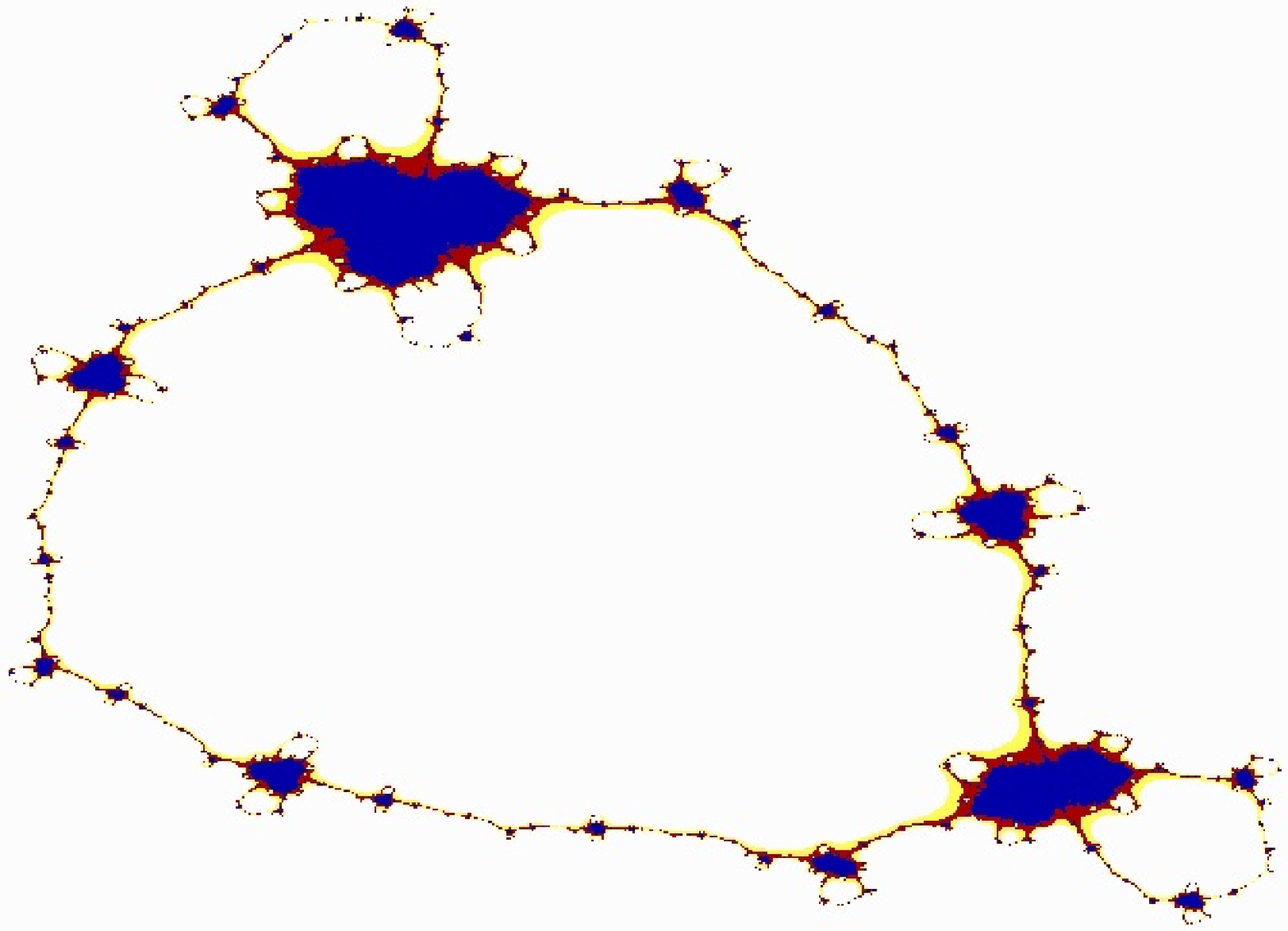}\\ \\
\text{Normalized model}&& \text{Canonical model}
\end{array}$$
\vspace{1cm}

\end{document}